\documentclass[preprint,11pt]{elsarticle}
\usepackage[latin1]{inputenc}
\usepackage[lmargin=2.5cm, rmargin=2.5cm,bottom=2.5cm,top=2.5cm]{geometry}

\usepackage[usenames]{color}
\usepackage{times}
\usepackage{amsfonts}
\usepackage{amsthm}
\usepackage{amsmath}
\usepackage{psfrag}
\usepackage{times}
\usepackage{latexsym,color}
\usepackage{amssymb}
\usepackage{graphicx}
\usepackage{epsfig}
\usepackage{gastex}

\font\bigmath=cmsy10 scaled \magstep 3
\font\smallrm=cmr8
\newcommand{\bigtimes}{\hbox{\bigmath \char'2}}
\newcommand{\Fin}{\hbox{\rm Fin}}

\newcommand{\nats}{{\mathbb N}}
\newcommand{\ints}{{\mathbb Z}}
\newcommand{\ben}{{\mathbb N}}
\newcommand{\TM}{{\mathbb T}}
\newcommand{\emp}{\emptyset}
\def\A{\mathbb{A}}

\newtheorem{theorem}{Theorem}[section]
\newtheorem{lemma}[theorem]{Lemma}
\newtheorem{corollary}[theorem]{Corollary}

\begin{document}

\begin{frontmatter}

\title{
On additive properties of sets defined by the Thue-Morse word }

\author[label1]{Michelangelo Bucci}
  \ead{michelangelo.bucci@utu.fi}

 \author[label2]{Neil Hindman\fnref{label5}}
  \ead{nhindman@aol.com}
  
  \author[label1,label3]{Svetlana Puzynina\fnref{label6}}
  \ead{svepuz@utu.fi}

   \author[label1,label4]{Luca Q. Zamboni\fnref{label7}}
  \ead{lupastis@gmail.com}

 \address[label1]{Department of Mathematics and Statistics \& FUNDIM,
University of Turku,  Finland.}
\address[label2]{Department of Mathematics, Howard University, USA.}
  \address[label3]{Sobolev Institute of Mathematics, Russia.}
 \address[label4]{Universit\'e de Lyon 1, France.}

\fntext[label5]{Partially supported by the US National Science Foundation
under grant DMS-1160566.}
\fntext[label6]{Partially supported by the Academy of Finland under grant 251371, by RFBR (grant
10-01-00424) and by RF President grant for young scientists
(MK-4075.2012.1).}
\fntext[label7]{Partially supported by a FiDiPro grant (137991) from the Academy of Finland and by
ANR grant {\sl SUBTILE}.}

\begin{abstract}
In this paper we study some additive properties of subsets of the set $\nats$ of positive integers:
A subset $A$ of $\nats$ is called {\it $k$-summable} (where $k\in\ben$) if
$A$ contains $\textstyle \big\{\sum _{n\in F}x_n \,| \,\emp\neq F\subseteq \{1,2,\ldots,k\} \big\}$ 
for some $k$-term sequence of natural numbers $\langle x_t\rangle_{t=1}^k$ satisfying uniqueness of
finite sums.
We say $A \subseteq \nats$ is \emph{finite FS-big} if $A$ is $k$-summable  for each positive integer 
$k$.  We say is $A \subseteq \nats$ is \emph{infinite FS-big} if for each positive 
integer $k,$  $A$ contains $\{\sum _{n\in F}x_n \,| \,\emp\neq 
F\subseteq \nats\hbox{ and }\#F\leq k\}$ for some
infinite sequence of natural numbers $\langle x_t\rangle_{t=1}^\infty$ satisfying uniqueness of
finite sums. We say $A\subseteq \nats $ is an 
{\it IP-set} if $A$ contains $\{\sum _{n\in F}x_n \,| \,\emp\neq F\subseteq \nats
\hbox{ and }\#F<\infty\}$ for some
infinite sequence of natural numbers $\langle x_t\rangle_{t=1}^\infty$.  By the Finite Sums Theorem
\cite{H}, the collection of all IP-sets is \emph{partition
regular}, i.e.,  if $A$ is an IP-set then for any finite partition of $A$, 
one cell of the partition is an IP-set.
Here we prove that the collection of all finite FS-big sets is also partition regular.
Let $\TM =011010011001011010010110011010\ldots $  denote the Thue-Morse word fixed by the 
morphism $0\mapsto 01$ and $1\mapsto 10$. For each factor $u$ of $\TM$ we consider
the set $\TM\big|_u\subseteq \nats$ of all occurrences of $u$ in $\TM$. 
In this note we characterize the sets $\TM\big|_u$  in terms of the additive properties 
defined above. Using the Thue-Morse word we show that the collection of all infinite FS-big sets is not partition regular.
\end{abstract}

\begin{keyword}Partition regularity, additive combinatorics, IP-sets,  Thue-Morse infinite word.
\MSC[2010] 68R15, 05D10
\end{keyword}

\end{frontmatter}
\section{Introduction}

A fundamental result in Ramsey theory, originally due to Issai Schur \cite{Sch}, 
states that given a finite partition of the natural numbers $\nats$, one cell of the 
partition contains two points $x,y$ and their sum $x+y$. 
An extension of Schur's Theorem, which we will call the {\it finite Finite Sums Theorem\/}
states that whenever $\nats$ is finitely partitioned, there exist arbitrarily large 
sets of numbers all of whose sums belong to the same element of the partition. 
The finite Finite Sums Theorem is an easy consequence of Rado's Theorem \cite{R}.
Given a finite sequence $\langle x_t\rangle_{t=1}^k$ or an infinite 
sequence $\langle x_t\rangle_{t=1}^\infty$ in $\nats$ we say that the sequence satisfies
{\it uniqueness of finite sums\/} provided that whenever $F$ and $H$ are finite
nonempty subsets of the domain of the the sequence and $\sum_{t\in F}\,x_t=\sum_{t\in H}\,x_t$,
one must have $F=H$.
For $k$ a positive integer, we say that a subset $A$ of $\nats$ is $k$-\emph{summable}, if $A$
contains all finite sums of distinct terms of some $k$-term sequence 
$\langle x_n\rangle_{n=1}^{k}$ of natural numbers satisfying uniqueness of finite sums. We say that 
$A\subseteq \nats$ is $k^\infty$-\emph{summable} if there
exists an infinite sequence $\langle x_n\rangle_{n=1}^{\infty}$ of natural numbers 
satisfying uniqueness of finite sums such
that $A$ contains all sums of at most $k$ distinct terms of $\langle x_n\rangle_{n=1}^{\infty}$.
As a consequence of the (infinite) Finite Sums Theorem, 
given any finite partition of $\nats$, one element of the partition is $k^\infty$-summable.

In this paper we consider three different families of subsets of $\nats$ each characterized by an 
additive property which may be regarded as an extension of the finite Finite Sums Theorem:  
finite FS-big,  infinite FS-big, and IP-sets.
A subset $A$ of $\nats$ is called \emph{finite FS-big} if it is $k$-summable for every positive integer $k$. 
A subset $A$ of
$\nats$ is  called \emph{infinite FS-big} if  it is $k^\infty$-summable for every positive integer $k$.
A subset $A$ of $\nats$ is called an
\emph{IP-set} if $A$ contains all finite sums of distinct terms of
some infinite sequence $\langle x_n\rangle_{n=1}^\infty $ of natural numbers.

A collection of sets  ${\mathcal S}$ is said to be {\it partition regular\/} if for each $A\in {\mathcal S},$ whenever $A$ is partitioned into finitely many sets, at least one set of the partition is in ${\mathcal S}.$    By the Finite Sums Theorem, the collection of all IP-sets is 
partition regular. Other examples of partition regular families are sets having positive upper 
density, and sets having arbitrarily long arithmetic progressions.  (This latter fact is an 
almost immediate consequence of van der Waerden's Theorem \cite{vdW}.  Assume $A\subseteq\nats$ contains
arbitrarily long arithmetic progressions. Let  $k,r\in\nats$, and let $A=\bigcup_{i=1}^rC_i$.  
By van der Waerden's
Theorem pick $n$ such that whenever $\{1,2,\ldots,n\}$ is partitioned into $r$ classes, one class contains
a length $k$ arithmetic progression.  Pick $a$ and $d$ in $\ben$ such that 
$\{a+d,a+2d,\ldots,a+nd\}\subseteq A$.
For $i\in\{1,\ldots,r\}$ let $B_i=\{t\in\{1,2,\ldots,n\}\,|\,a+td\in C_i\}$.  Pick $i$, $b$, and $c$ such that
$\{b+c,b+2c,\ldots,b+kc\} \subseteq B_i$.  Then $\{a+bd+cd,a+bd+2cd,\ldots,a+bd+kcd\}\subseteq C_i$.)

We shall show in Section 2 that the collection of all finite FS-big subsets of $\nats$ is partition regular (see Theorem~\ref{FSBPR}).
In contrast,  for any fixed value of $k$, the property of being $k$-summable or $k^\infty$-summable is 
not partition regular.  For example, the set $A=\{n\in\ben\,|\,n\not\equiv 0\hbox{ mod }3 \}$ 
is clearly $2^\infty$-summable. On the other hand $A=A_1\cup A_2$
where $A_1=\{n\in\ben\,|\,n\equiv 1 \mod 3 \}$ and $A_2=\{n\in \ben\,|\,n\equiv 2 \mod 3 \},$ 
and neither $A_i$ is $2$-summable. Nevertheless, for each fixed $k$ we could consider the set
\[\mathcal{R}^\infty(k)=\{ A\subseteq \nats\,|\, \mbox{whenever }r\in\nats
\mbox{ and }A=\bigcup_{i=0}^rA_i,\,\,\exists \, 0\leq i\leq r\,\mbox{such that} \,A_i \,
\mbox{is $k^\infty$-summable}\}\]
Then each $\mathcal{R}^\infty(k)$ is non-empty.
In fact every IP-set belongs to $\mathcal{R}^\infty(k)$. It is a difficult open question of Imre Leader's
\cite[Question 8.1]{BR} whether there is any member of $\mathcal{R}^\infty(2)$ which is not an IP-set.
In general, the question of determining whether a given subset $A\subseteq
\nats$ is in $\mathcal{R}^\infty(k)$ or is an IP-set  is typically
quite difficult, even if for every $A,$ either $A$ or its
complement belongs to $\mathcal{R}^\infty(k)$ or is an IP-set. 

In this note we 
focus on sets $A$ which are defined in terms of the binary expansions of its elements. In this respect it is natural to consider the {\it Thue-Morse infinite word} $\TM=t_0t_1t_2t_3\ldots \in \{0,1\}^\omega$
where $t_n$ is defined as the sum modulo $2$ of the digits  in the binary expansion of
$n.$  The Thue-Morse word is  $2$-automatic~\cite{AS2}: In fact $t_n$ is computed by feeding the binary expansion of $n$ in the deterministic finite automata depicted in Figure~\ref{Gtau}. Starting from the initial state labeled $0,$ we read the binary expansion of $n$  starting from the most significant digit. Then $t_n$ is the corresponding terminal state. For example, the binary representation of $13$ is $1101$ and the path $1101$ starting at $0$ terminates at vertex $1.$ Whence $t_{13}=1.$ 

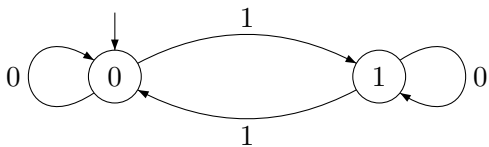
\begin{figure}[h]
  \begin{center}
    \unitlength=4pt
    \begin{picture}(25, 10)(0,-5)
      \gasset{Nw=5,Nh=5,Nmr=2.5,curvedepth=4}
      \thinlines
      \node[Nmarks=i,iangle=90](A1)(0,0){$0$}
      \drawloop[loopangle=180](A1){$0$}
      \node[](A2)(25,0){$1$}
      \drawloop[loopangle=0](A2){$0$}
      \drawedge(A1,A2){$1$}
      \drawedge(A2,A1){$1$}
    \end{picture}
  \end{center}
  \caption{The Thue-Morse automaton}
\end{figure}\label{Gtau}

The origins of $\TM$ go back to the beginning of the last century with the works of  the Norwegian mathematician Axel Thue \cite{Th1, Th2}. Thue noted that every binary word of length four contains a square, that is two consecutive equal blocks $XX.$ He then asked whether it was possible to find an infinite word on $3$ distinct symbols which avoided squares.
He also asked whether there exists an infinite binary word without cubes, that is with no three consecutive equal blocks.  Thue showed that in each case the answer is positive and constructed this very special infinite word $\TM$ to produce the desired words. In fact the word $\TM$ contains no fractional power greater than $2,$ i.e., contains no word of the form $XXX'$ where $X'$ is a prefix of $X.$
Thue's work originally appeared in an obscure Norwegian journal and for many years remained largely unknown and unappreciated.

A few years later in the 1920s, Marston Morse and Gustav Hedlund \cite{MoHe1, MoHe2} were pioneering a new branch of mathematics known as Symbolic Dynamics, inspired by the study of various classical  dynamical systems dating back to Newton.  
The basic idea of symbolic dynamics consists in dividing up the set of possible states into a finite number of pieces. By discretizing both space and time, one could model a dynamical system $(X,T)$ by a space consisting of infinite words of abstract symbols, each symbol corresponding to a state of the system, and a shift operator corresponding to the dynamics. Thus from this point of view, the orbits of motion are described as {\it symbolic trajectories} or {\it flows}. A periodic orbit would give rise to a periodic infinite word, while an aperiodic orbit would correspond to an aperiodic infinite word.

Curiously enough, these foundational works of Morse and Hedlund exhibited strong ties with the earlier work of Thue. This connection stems through the use of infinite words to describe infinite geodesic curves on a surface of negative curvature. And so, the word $\TM$ originally defined by Thue to study combinatorial properties of words was
rediscovered in 1921 by Morse \cite{Mo} in connection with differential geometry.
He proved that every surface of negative curvature having at least two normal segments, admits a continuum of recurrent aperiodic geodesics.

An alternative definition of the Thue-Morse word which will be useful to us is in terms of the  
morphism $\tau :\{0,1\}\rightarrow \{0,1\}^*$ given by $0\mapsto 01$ and $1\mapsto 10.$
More precisely, iterating $\tau$ on the symbol $0$ gives
\[0\mapsto 01 \mapsto 0110 \mapsto 01101001 \mapsto  0110100110010110 \mapsto \cdots.\]
In general, $\tau^{n+1}(0)=\tau^n(0)\overline{\tau^n(0)}$ where $\overline{\tau^n(0)}$ is obtained from 
$\tau^n(0)$ by exchanging $0$'s and $1$'s. In particular, since $\tau^n(0)$ is a prefix of $\tau^{n+1}(0),$
the sequence $(\tau^n(0))_{n\geq 0}$ tends in the limit  to the infinite word 
\[\TM =0110100110010110100101100110100110010110\ldots  \] 
For more background and information on the 
Thue-Morse word we refer the reader to \cite{AS1} or \cite{AS2}.

 Let $\overline{\TM}$ denote the word obtained from $\TM$ by exchange of 
$0$'s and $1$'s, i.e., $\overline{\TM}$ is the fixed point of the Thue-Morse morphism beginning in $1$.
We consider subsets of $\nats$ defined by the Thue-Morse word via occurrences of
its factors. More precisely, writing $\TM=t_0t_1t_2\ldots $ with $t_i\in \{0,1\},$
for each factor $u$  of $\TM$ we set
\[\TM\big|_{u}=\{ n\in \nats \,|\, t_nt_{n+1}\ldots t_{n+|u|-1}=u\}.\]
In other words, $\TM\big|_u$ denotes the set of all occurrences
of $u$ in $\TM$. The main result 
of this note is to obtain a full characterization of each of the sets
$\TM\big|_u$ in terms of the three additive properties defined above. We show that factors of the 
Thue-Morse word can be split into three classes: one corresponding to factors $u$ for which
$\TM\big|_u$ is an IP-set; these factors are precisely all prefixes of $\TM$. The second class 
consists of all factors $u$ such that $\TM\big|_u$ is infinite FS-big but not an IP-set; 
this corresponds to all prefixes of $\overline{\TM}$. Finally, for all remaining factors 
$u$ of $\TM$, the set $\TM\big|_u$ is not $3$-summable, and in some cases not even $2$-summable (see Theorem~\ref{main}). 
We also show that the set $\TM\big|_1$ may be partitioned into two cells neither of 
which is $2^\infty$-summable (see Lemma~\ref{TM1}). Thus, the collection of all infinite FS-big sets is not partition
regular (see Corollary~\ref{notpartreg}). As pointed out to us by the referees of this paper, this  latter point may be proved independently of the Thue-Morse word (either directly using the binary representation of digits without reference to $\TM$, or via other digital representations of the integers). Our use of $\TM$
is merely one of convenience as it provides a uniform framework on which to investigate the various additive properties defined above. 

\noindent We conclude this introduction with some notation that we will be using.
We denote the set of all $k$-summable subsets of $\nats$ by $\Sigma_k$ and  the set of all finite FS-big  subsets of $\nats$
by $\Sigma$. Thus, $\Sigma= \bigcap_{k\geq 1} \Sigma_k$.
We denote the set of all $k^\infty$-summable sets by $\Sigma_k^\infty$ and the set of all infinite FS-big sets by $\Sigma^\infty$ so that $\Sigma^\infty= \bigcap_{k\geq
1} \Sigma_k^{\infty}$. 
It is immediate that $\Sigma_{k+1}^{\infty}\subseteq \Sigma_k^{\infty}$ and
$\Sigma_{k}^\infty\subseteq \Sigma_k$.

We have seen that $\{n\in\ben\,|\,n\not\equiv 0\ \mbox{mod}\,3 \}\in \Sigma^\infty_2\setminus\Sigma_3$.
More generally, for each $k>2$ we have that $\{n\in\ben\,|\,n\not\equiv 0\ \mbox{mod}\,k \}\in \Sigma^\infty_{k-1}\setminus\Sigma_k.$ This follows immediately from the following simple lemma which is likely a well known fact but the authors were unable to find it in the literature. We thus include  a proof here for the sake of completeness.

\begin{lemma}
Given any set $S$ of $k$ nonnegative integers, some subset $L$ of $S$ sums up to $0$ modulo $k,$ i.e., $\sum_{x\in L}x \equiv 0\ \mbox{mod}\,k$ for some $L\subseteq S.$ 
\end{lemma}

\begin{proof} Equivalently, given a $k$-term sequence $\langle x_i\rangle _{i=1}^k$ in the cyclic group $\ints_k$ of order $k,$ we will show that some subsequence sums up to $0.$  To see this, for each $0\leq i\leq k,$ define sets $C_i\subseteq \ints_k$ recursively as follows:  $C_0=\{0\}$  
and for $i\geq 0,$ set
$C_{i+1}=C_i \,\bigcup \, (C_i+ x_{i+1}),$ in other words  \[C_{i+1}= \{0\} \, \bigcup \, \{\sum _{i\in F}x_i\,|\, F\subseteq \{1,2\,\ldots ,i+1\}\}.\] We claim that $0\in C_i+ x_{i+1}$ for some $0\leq i\leq k-1,$ i.e., some subsequence of   $\langle x_j\rangle _{j=1}^k$ sums to $0.$ In fact, for each  $0\leq i\leq k-1,$ 
if $ 0 \notin C_i+ x_{i+1}$ then  $\#C_{i+1}>\#C_i$ (since $0\in C_i).$
Thus if for every $0\leq i\leq k-1$ we had that $ 0 \notin C_i+ x_{i+1},$ then $\#C_k \geq k+\#C_0=k+1,$ a contradiction (since $C_k \subseteq \ints_k).$ 
\end{proof}

\noindent{\bf Acknowledgements:} The authors would like to express their gratitude to the two referees of this paper. In particular one of the referees suggested a simplification to the proof of item (1) of Theorem~\ref{main} which we decided to use, as well as an alternative simple but clever proof that the collection of all infinite FS-big sets is not partition regular which we included as one of two proofs of Corollary~\ref{notpartreg}.

\section{Finite FS-big sets}

In this section we prove that the collection of all finite FS-big sets is partition regular (see Theorem~\ref{FSBPR} below). Surprisingly the authors were unable to find a proof of this fact in the existing literature. Thus we take this opportunity to present two different derivations: Our first proof is a straightforward application of the so-called finite Finite Unions Theorem (see Theorem~\ref{FFU}  below) and is by now quite routine to experts in Ramsey theory.  Our second proof uses a clever argument  suggested to us by Imre Leader which establishes Theorem \ref{FFU}
using only the finite Finite Sums Theorem (Theorem \ref{FFS} ) and Ramsey's Theorem \cite{Ram}.

Throughout this section we will
use the notation $\Fin(A)$ for the set of all
non-empty finite subsets of $A$ and let $FS(\langle x_t\rangle_{t\in A})=
\{\sum_{t\in F}\,x_t\,|\,F\in\Fin(A)\}$.

We first observe that we had some choices to make when we defined $k$-summable.
That is, we could have defined $A$ to be $k$-summable${}_1$ if there is a 
sequence $\langle x_t\rangle_{t=1}^k$ such that $FS(\langle x_t\rangle_{t=1}^k)\subseteq A$;
 we could have defined $A$ to be $k$-summable${}_2$ if there is an increasing 
sequence $\langle x_t\rangle_{t=1}^k$ such that $FS(\langle x_t\rangle_{t=1}^k)\subseteq A$;
and  we could have defined $A$ to be $k$-summable${}_3$ if there is a 
sequence $\langle x_t\rangle_{t=1}^k$ satisfying uniqueness of finite sums 
such that $FS(\langle x_t\rangle_{t=1}^k)\subseteq A$.  (We actually chose
$k$-summable${}_3$ because it generalizes most naturally to arbitrary semigroups as in
Theorem \ref{FSBPR}.)  These notions are progressively strictly stronger.  For
example if $k>1$, $\{1,2,\ldots,k\}$ is $k$-summable${}_1$ but not $k$-summable${}_2$.
And if $k>1$, $\{1,2,\ldots,\frac{k^2+k}{2}\}$ is $k$-summable${}_2$ but not $k$-summable${}_3$.
However, for the notion of finite FS-big subsets of $\nats$, 
it does not matter which choice was made for $k$-summable.
The reason is that for each $k$ there is some $m$ such that if $A$ is an $m$-summable${}_1$
subset of $\nats$,
then $A$ is $k$-summable${}_2$. Similarly, for each $k$ there is some $m$ such that if 
$A$ is an $m$-summable${}_2$ subset of $\nats$, then $A$ is $k$-summable${}_3$.

The main key for proving that finite FS-big sets are partition regular is the finite Finite Unions
Theorem.  The first proof that we will present, and much the simpler of the two, uses
a standard compactness argument and the infinite Finite Unions Theorem.

\begin{theorem}[Infinite Finite Unions Theorem]\label{IFU}  Let $r\in\ben$.
If $\Fin(\ben)=\bigcup_{i=1}^r{\mathcal F}_i$, then there exist $i\in\{1,2,\ldots,r\}$ and 
a sequence $\langle F_t\rangle_{t=1}^\infty$ in $\Fin(\ben)$ such that for each
$t\in\ben$, $\max F_t<\min F_{t+1}$ and for each $H\in\Fin(\ben)$, 
$\bigcup_{t\in H}\,F_t\in{\mathcal F}_i$.\end{theorem}

\begin{proof} This is actually stated in \cite{H}.  A much easier proof is 
in \cite[Corollary 5.17]{HS}. It is an immediate corollary of the (infinite) Finite Sums Theorem,
because, given any sequence $\langle x_t\rangle_{t=1}^\infty$ in $\ben$ one may choose
a sequence $\langle F_n\rangle_{n=1}^\infty$ in $\Fin(\ben)$ such that for each
$n\in\ben$, $\max F_n<\min F_{n+1}$ and for each $n$ and $l$ in $\ben$,
if $2^l\leq\sum_{t\in F_n}x_t$, then $2^{l+1}$ divides $\sum_{t\in F_{n+1}}x_t$.
(That is, the maximum of the binary support of $\sum_{t\in F_n}x_t$ is less than
the minimum of the binary support of $\sum_{t\in F_{n+1}}x_t$.)
\end{proof}

\begin{theorem}[Finite Finite Unions Theorem]\label{FFU} Let $r,k\in\ben$.  
There is some $m\in\ben$ such that whenever 
$\Fin(\{1,2,\ldots,m\})=\bigcup_{i=1}^r{\mathcal F}_i$, there exist $i\in\{1,2,\ldots,r\}$ and 
a sequence $\langle F_t\rangle_{t=1}^k$ in $\Fin(\{1,2,\ldots,m\})$ such that for each
$t\in\{1,2,\ldots,k-1\}$, if any, $\max F_t<\min F_{t+1}$ and for each $H\in\Fin(\{1,2,\ldots,k\})$, 
$\bigcup_{t\in H}\,F_t\in{\mathcal F}_i$.\end{theorem}

\begin{proof}  Suppose not.  For each $m\in\ben$ pick a function 
$\psi_m:\Fin(\{1,2,\ldots,m\})\to\{1,2,\ldots,r\}$ with the property that there do not exist
$i\in\{1,2,\ldots,r\}$ and a sequence $\langle F_t\rangle_{t=1}^k$ in $\Fin(\{1,2,\ldots,m\})$ such that for each
$t\in\{1,2,\ldots,k-1\}$, if any, $\max F_t<\min F_{t+1}$ and for each $H\in\Fin(\{1,2,\ldots,k\})$, 
$\psi_m(\bigcup_{t\in H}\,F_t)=i$.  Define $\sigma_m:\Fin(\ben)\to\{1,2,\ldots,r\}$ by
$\sigma_m(F)=\psi_m(F)$ if $F\subseteq \{1,2,\ldots,m\}$ and $\sigma_m(F)=1$ otherwise.

Give $\{1,2,\ldots,r\}$ the discrete topology and let 
$X=\bigtimes_{F\in\hbox{\smallrm Fin}(\ben)}\{1,2,\ldots,r\}$ with
the product topology.  Then $X$ is compact and $\langle \sigma_m\rangle_{m=1}^\infty$ is a sequence
in $X$ so pick a cluster point $\varphi$ of $\langle \sigma_m\rangle_{m=1}^\infty$.  Pick by Theorem
\ref{IFU}, $i\in\{1,2,\ldots,r\}$ and 
a sequence $\langle F_t\rangle_{t=1}^\infty$ in $\Fin(\ben)$ such that for each
$t\in\ben$, $\max F_t<\min F_{t+1}$ and for each $H\in\Fin(\ben)$, 
$\varphi(\bigcup_{t\in H}\,F_t )=i$.  Let  $$U=\big\{\mu\in X\,|\,\mu(F_i)=\varphi(F_i)\hbox{ for all }
i\in\{1,2,\ldots,k\}\big\}\,.$$  Then $U$ is a neighborhood of $\varphi$ in $X$ so pick $m>\max F_k$
such that $\sigma_m\in U$.  Then for each $H\in\Fin(\{1,2,\ldots,k\})$,
$\psi_m(\bigcup_{t\in H}\,F_t)=\sigma_m(\bigcup_{t\in H}\,F_t)=\varphi(\bigcup_{t\in H}\,F_t)=i$,
a contradiction.\end{proof}

The definition of FS-big makes sense in an arbitrary semigroup $(S,+)$.  (Even though we are
writing the semigroup additively, we are not assuming commutativity, so we need to specify that
the sums are taken in increasing order of indices.) The reader should be cautioned that
an arbitrary semigroup might have no nontrivial sequences satisfying uniqueness of finite
sums, in which case $\Sigma=\emptyset$.  However, if $S$ is cancellative, then by
\cite[Lemma 6.31]{HS}, any infinite subset of $S$ contains a sequence satisfying uniqueness of
finite products.

\begin{theorem}\label{FSBPR} Let $(S,+)$ be a semigroup. The collection $\Sigma$ of all 
finite FS-big subsets of $S$ is partition regular. \end{theorem}

\begin{proof} Suppose  $A\subseteq S$ is finite FS-big and   
$A=\bigcup_{i=1}^rB_i$ for some $r\in \nats.$
Let $k\in\nats$.  We shall show that there are some $i\in\{1,2,\ldots,r\}$
and some $\langle x_t\rangle_{t=1}^k$ satisfying uniqueness of finite
sums such that $FS(\langle x_t\rangle_{t=1}^k)\subseteq B_i$.  
By the pigeon hole principle, there is thus one $i$ which contains
such a set for arbitrarily large $k$, and thus for all $k$.

By Theorem \ref{FFU} pick $m\in\ben$ such that whenever 
$\Fin(\{1,2,\ldots,m\})=\bigcup_{i=1}^r{\mathcal F}_i$, then there exist $i\in\{1,2,\ldots,r\}$ and 
a sequence $\langle F_t\rangle_{t=1}^k$ in $\Fin(\{1,2,\ldots,m\})$ such that for each
$t\in\{1,2,\ldots,k-1\}$, if any, $\max F_t<\min F_{t+1}$ and for each $H\in\Fin(\{1,2,\ldots,k\})$, 
$\bigcup_{t\in H}\,F_t\in{\mathcal F}_i$.  Since $A$ is finite FS-big we may pick 
$\langle y_t\rangle_{t=1}^m$ satisfying uniqueness of finite sums
with $FS(\langle y_t\rangle_{t=1}^m)\subseteq A$.  For each
$i\in\{1,2,\ldots,r\}$ let ${\mathcal F}_i=\{H\in\Fin(\{1,2,\ldots,m\})\,|\,
\sum_{t\in H}\,y_t\in B_i\}$.  Pick $i\in\{1,2,\ldots,r\}$ and 
a sequence $\langle F_t\rangle_{t=1}^k$ in $\Fin(\{1,2,\ldots,m\})$ such that for each
$t\in\{1,2,\ldots,k-1\}$, if any, $\max F_t<\min F_{t+1}$ and for each $H\in\Fin(\{1,2,\ldots,k\})$, 
$\bigcup_{t\in H}\,F_t\in{\mathcal F}_i$.  For $n\in\{1,2,\ldots,k\}$ let
$x_n=\sum_{t\in F_n}y_t$. 
Then since $\max F_t<\min F_{t+1}$ when $t<k$, if
$H\in\Fin(\{1,2,\ldots,k\})$ and
$K=\bigcup_{n\in H}F_n$, then $\sum_{n\in H}x_n=\sum_{t\in K}y_t\in B_i$.
Further it is an easy exercise to show that, since 
$\langle y_t\rangle_{t=1}^m$ satisfies uniqueness of finite sums,
so does $\langle x_t\rangle_{t=1}^k$.
\end{proof}

Notice that, since we did the above proof for an arbitrary semigroup, it would
not be good enough to have $F_t\cap F_l=\emptyset$ when $t\neq l$.  For example,
if $F_1=\{1,3\}$, $F_2=\{2\}$, $H=\{1,2\}$, and $K=\bigcup_{n\in H}F_n$, then $K=\{1,2,3\}$. Thus
$\sum_{n\in H}x_n=y_1+y_3+y_2$ which need not equal $y_1+y_2+y_3=\sum_{t\in K}y_t$.

As we remarked earlier, the finite Finite Sums Theorem, (sometimes called Folkman's Theorem),
has been known, or at least easily knowable, since the proof of Rado's Theorem \cite{R} was
published in 1933.  

\begin{theorem}[Finite Finite Sums Theorem]\label{FFS} Let $k,r\in\ben$.  There exists
$m\in\ben$ such that whenever $\{1,2,\ldots,m\}=\bigcup_{i=1}^r B_i$, there
exist $i\in\{1,2,\ldots,r\}$ and a sequence $\langle x_t\rangle_{t=1}^k$ such that
$FS(\langle x_t\rangle_{t=1}^k)\subseteq B_i$.\end{theorem}

\begin{proof} This is an easy consequence of Rado's Theorem.  See
\cite[Exercise 15.3.1]{HS}.\end{proof}

While it is immediate that Theorem \ref{FFU} implies Theorem \ref{FFS} (by means
of the binary support of integers), it is by no means obvious that one
can derive Theorem \ref{FFU} from Theorem \ref{FFS}.  We are grateful to 
Imre Leader for providing an argument which establishes Theorem \ref{FFU}
using only Theorem \ref{FFS} and Ramsey's Theorem \cite{Ram}.

\begin{lemma}\label{krs} Let $k,r,s\in\ben$ with $k\leq s$.  There 
exists $m\in\ben$ such that whenever $A$ is a set with $\#A=m$, and
$\Fin(A)=\bigcup_{i=1}^r{\mathcal F}_i$, there exist $\varphi:\{1,2,\ldots,k\}
\to \{1,2,\ldots,r\}$ and $B\subseteq A$ with $\#B=s$ such that for
all $C\in \Fin(B)$, if $t=\#C$ and $t\leq k$, then 
$C\in{\mathcal F}_{\varphi(t)}$.\end{lemma}

\begin{proof}  We proceed by induction on $k$ (for all $r$ and all $s\geq k$).
For $k=1$ the conclusion is an immediate consequence of the pigeon hole principle.

Now assume that the lemma holds for $k$.  Let $r,s\in\ben$ be given with
$s\geq k+1$.  By Ramsey's Theorem pick $n\in\ben$ such that if
$\#D=n$ and $\{C\subseteq D\,|\,\#C=k+1\}\subseteq\bigcup_{i=1}^r{\mathcal F}_i$,
then there exist $i\in\{1,2,\ldots,r\}$ and $B\subseteq D$ such that $\#B=s$
and $\{C\subseteq B\,|\,\#C=k+1\}\subseteq {\mathcal F}_i$.

Pick $m$ as guaranteed by the induction hypothesis for $k$, $r$, and $n$ (with $n$
replacing $s$) and let $\#A=m$.  Assume that $\Fin(A)=\bigcup_{i=1}^r{\mathcal F}_i$.
Pick $\varphi:\{1,2,\ldots,k\}
\to \{1,2,\ldots,r\}$ and $D\subseteq A$ with $\#D=n$ such that for
all $C\in \Fin(D)$, if $t=\#C$ and $t\leq k$, then 
$C\in{\mathcal F}_{\varphi(t)}$. Then
$\{C\subseteq D\,|\,\#C=k+1\}\subseteq\bigcup_{i=1}^r{\mathcal F}_i$,
so pick  $\varphi(k+1)\in\{1,2,\ldots,r\}$ and $B\subseteq D$ such that $\#B=s$
and $\{C\subseteq B\,|\,\#C=k+1\}\subseteq {\mathcal F}_\varphi(k+1)$.\end{proof}

{\it Second proof of Theorem\/} \ref{FFU}  Let $k,r\in\ben$ and pick
by Theorem \ref{FFS} $s\in\ben$ such that whenever $\{1,2,\ldots,s\}=\bigcup_{i=1}^r C_i$,
there exist $x_1,x_2,\ldots,x_k$ and $i\in\{1,2,\ldots,r\}$ such that
$FS(\langle x_t\rangle_{t=1}^k)\subseteq C_i$.
Let $k'=s$ and pick $m$ as guaranteed by Lemma \ref{krs} for $r$, $k'$, and $s$.
Let $\Fin(\{1,2,\ldots,m\})=\bigcup_{i=1}^r{\mathcal F}_i$.
Pick $\varphi:\{1,2,\ldots,s\}
\to \{1,2,\ldots,r\}$ and $B\subseteq \{1,2,\ldots,m\}$ with $\#B=s$ such that for
all $C\in \Fin(B)$, if $t=\#C$ and $t\leq s$, then 
$C\in{\mathcal F}_{\varphi(t)}$.  Pick $x_1,x_2,\ldots,x_k$ and $i\in\{1,2,\ldots,r\}$ such that
$\varphi[FS(\langle x_t\rangle_{t=1}^k)]=\{i\}$.  Pick
$\langle F_t\rangle_{t=1}^k$ with $\max F_t<\min F_{t+1}$ for all $t<k$
such that $\#F_t=x_t$, which one can do since $\sum_{t=1}^kx_t\leq s$.\qed

\section{Sets defined by the Thue-Morse word}

In this section we define a class of  subsets of $\nats$ defined by
the occurrences of factors in the Thue-Morse word.

\begin{theorem}\label{main}  Let $u$ be a factor of the Thue-Morse word $\TM =011010011001011010\ldots $.
Then
\begin{enumerate}

\item If $u$ is a prefix of $\TM$ then  $\TM\big|_u$ is an IP-set.

\item If $u$ is a prefix of $\overline{\TM}$ then $\TM\big|_u$ is infinite FS-big but is not an IP-set.

\item If $u$ is neither a prefix of $\TM$ nor a prefix of $\overline{\TM}$  then
$\TM \big|_u$ is not $3$-summable. Moreover $\TM \big|_u$ is $2$-summable if and only 
$u$ is a prefix of $\tau^n(010)$ or of $\tau^n(101)$ for some $n\geq 0$.

\end{enumerate}
\end{theorem}

Before we begin with the proof of Theorem~\ref{main} we introduce some useful notation: 
For each positive integer
$n$ we will denote the binary expansion of $n$ by $[n]_2$, i.e., if 
$n=r_k2^k + r_{k-1}2^{k-1} +\ldots +r_02^0$ with
$r_k=1$ and $r_i\in\{0,1\}$ we write $[n]_2=r_kr_{k-1}\ldots r_0$. We define the {\it support} of 
$n$, denote ${\rm supp}(n)$ by ${\rm supp}(n)=\{i\in \{0,1,\ldots ,k\}\,|\, r_i=1\}$.
For instance, ${\rm supp}(19)=\{0,1,4\}$.
Thus
\[t_n=0 \Leftrightarrow \#{\rm supp}(n) \,\mbox{is even}.\]
Finally, for each length $n$ we denote by $\textrm{pref}_{n}\TM$ the prefix of $\TM$ of length $n$.

\begin{proof} [Proof of Theorem \ref{main}, part 1.] It follows from the definition of the Thue-Morse word $\TM$ that
 if $u=u_1u_2\ldots u_k \in \{0,1\}^k$ is a 
factor of $\TM,$ then $m\in \TM\big|_u$ if and only if 
 \[\#{\rm supp}(m+j)\equiv u_{j+1}\bmod 2\]
 for each $0\leq j\leq |u|-1.$ 
 Thus, if $2^n>m +|u|-1$ then  $\#{\rm supp}(2^{n+1}+2^n+m+j)= 2 + \#{\rm supp}(m+j)$ for each $0\leq j\leq |u|-1$ from which it follows that  $2^{n+1}+2^n+m\in \TM\big|_u.$
 Hence if $0\in \TM\big|_u$ (equivalently if $u$ is a prefix of $\TM),$ then there is a sequence of positive integers of the form $2^n+2^{n+1}$ whose finite sums are all in $\TM\big|_u.$
Thus, $\TM\big|_u$ is an IP-set. This completes the proof of 1.

We will need the following lemma in the proof of 2.:

\begin{lemma}\label{old} Let $i,j,k$ and $r$ be positive integers with 
$r$ odd and $j\leq k-2$. If $[r]_2=r_jr_{j-1}\ldots r_0$  then
\[\#{\rm supp}\big(r2^i(2^k-1)\big)=k.\]

\end{lemma}

\begin{proof} Since $\#{\rm supp}\big(r2^i(2^k-1)\big)=\#{\rm supp}\big(r(2^k-1)\big)$ 
it suffices to show that
$\#{\rm supp}\big(r(2^k-1)\big)=k$.

 We have that $[r2^k]_2=r_lr_{l-1}\ldots r_00^k$. Thus $[r2^k-1]_2=r_lr_{l-1}\ldots r_01^{k-1}$.

Since $l\leq k-2$ we have
\[\#{\rm supp}\big(r(2^k-1)\big)= \#{\rm supp}(r2^k-1 -r +1)=\#{\rm supp}(r) +k-1 -\#{\rm supp}(r) +1=k\]
where the last $+1$ term comes from the fact that $r_0=1$ since $r$ is odd.
\end{proof}

{\it Proof of Theorem \ref{main}, part 2.} We first note that those $n$'s for which
$\#{\rm supp}(n) $ is odd and $[n]_2$ ends in
 $0^{l}$ correspond to occurrences of
$\textrm{pref}_{2^{l}}\overline{\TM}$.  

Let $u$ be a prefix of $\overline{\TM}$ and $k$ a positive integer. To prove that $\TM_u\in
\Sigma_{2k-1}^\infty$ consider the sequence
$\langle x_n\rangle_{n=0}^{\infty}$ of numbers whose binary representation is given by
$$[x_n]_2= 110^{2n+j}1^{2k-1}0^{l} \,\mbox{where}\, j= \lceil\log_2
(2k-1)\rceil\,\mbox{and}\, l=\lceil \log_2 |u| \rceil\, .$$ Consider any $r\leq
2k-1$ distinct numbers $x_{n_i}$ and consider their sum

$$\sum_{i=1}^{r}x_{n_i}=\sum_{i=1}^{r}(2^{2n_i+2k-1+l+j}+2^{2n_i+2k+l+j})+
r(2^{2k-1}-1)2^{l}.$$

By Lemma~\ref{old} it follows that  $\#{\rm supp}(r(2^{2k-1}-1)2^{l})=2k-1$ and hence  
that $\#{\rm supp}(\sum_{i=1}^{r}x_{n_i})=2k-1+2r$. As this is an odd number,
and $[\sum_{i=1}^{r}x_{n_i}]_2$ ends in at least $l=\lceil\log_2{|u|}\rceil$ many $0$'s, it follows that
$\sum_{i=1}^{r}x_{n_i}$ is an occurrence of $u$.

Next we will prove that if $u$ is a prefix of $\overline{\TM}$ then
$\TM\big|_u$ is not an IP-set.
We will make use of the following lemma:

\begin{lemma}\label{TM1} There exists a partition of the set $\TM\big|_1$ into two sets 
neither of which is in $\Sigma_2^\infty$.
\end{lemma}

\begin{proof}
Consider the partition $\TM\big|_1=A_0 \cup A_1$ defined as
follows: Let  $A_0$ be the set of all $n\in \TM\big|_1$ such that
the ${\rm min}\big({\rm supp}(n)\big)$  is even, and let $A_1$
be the set of all $n\in \TM\big|_1$ such that the ${\rm
min}\big({\rm supp}(n)\big)$  is odd. For instance,
$25=2^4+2^3+2^0$, and hence the least nonzero digit is in position
$0$, so $25\in A_0$. We will show that neither $A_i$ is in
$\Sigma_2^\infty$.  Fix $i\in \{0,1\}$ and suppose to the contrary
that $A_i$ is in $\Sigma_2^\infty$, i.e., there is an infinite 
sequence $\langle x_n\rangle_{n=1}^{\infty}$ in $A_i$ 
satisfying uniqueness of finite sums such that for every
$n\neq m$ we have $x_n+x_m\in A_i$. Note first that for each
$n>1$ we have  ${\rm supp}(x_n)\cap {\rm supp}(x_1)\neq
\emptyset$.  Otherwise  $\#{\rm supp}(x_1+x_n)$ would be even.
Therefore, there exists a positive constant $M$ such that ${\rm
min}\big({\rm supp}(x_n)\big)\leq M$ for each $n\in\ben$. By the
pigeon hole principle there exists a positive integer $r$ and an
infinite subsequence $x_{n_1},x_{n_2},\ldots $ of the sequence
$\langle x_n\rangle_{n=1}^{\infty}$ such that ${\rm min}\big({\rm
supp}(x_{n_j})\big)=r$ for each $j\in\ben$. Again by the pigeon
hole principle there exists  infinitely many of the $x_{n_j}$
whose binary expansions also agree in position $r+1$. Thus there
exists $n\neq m$ such that ${\rm min}\big({\rm
supp}(x_n)\big)={\rm min}\big({\rm supp}(x_m)\big)=r$ and
such that $r+1 \in {\rm supp}(x_n)$ if and only if $r+1 \in {\rm
supp}(x_m)$. It is readily verified that  ${\rm min}\big({\rm
supp}(x_n+x_m)\big)=r+1$. Hence $x_n+x_m\in A_{1-i}$.
\end{proof}

\textbf{Remark.} It is not difficult to see that the sets $A_0$
and $A_1$ from the proof of Lemma \ref {TM1} are both finite FS-big. So
they provide examples of sets which are finite FS-big but not
$2^\infty$-summable.\\

It follows from the above lemma that $\TM\big|_1$ is not an IP-set. In fact, the  property of being an
IP-set is partition regular, so for any finite partition of $\TM\big|_1$
one element of the partition must be an IP-set and in particular must be
in $\Sigma_2^\infty$. But this contradicts Lemma \ref {TM1}. 
Let $u$ be a prefix of $\overline{\TM}$.  Since $\TM\big|_u \subseteq \TM\big|_1$ it follows  
that $\TM\big|_u$  is not an IP-set.

\emph{Proof of Theorem \ref{main}, part 3.} We will make use of the following
lemma:

\begin{lemma}\label{thm:lemmaTM} Let $u$ be a factor of $\TM$ which is neither a prefix of 
$\TM$ nor a prefix of $\overline{\TM}$. Then there exists a nonnegative integer $k$ such that 
one of the two following properties holds:
\begin{enumerate}
\item\label{eqn:point1} For each $x\in \TM \big|_u$, $[x]_2$ ends in $10^{k}$.
\item\label{eqn:point2} For each $x\in \TM \big|_u$ either $[x]_2$ ends in $110^{k}$ or 
in $10^{k+1}$ and both cases happen. Furthermore, $u$ is a prefix of $\tau^n(aba)$ for some 
nonnegative integer $n$, with $a,b$ distinct letters.
\end{enumerate}
\end{lemma}

\begin{proof}  Let $u$ be a factor of $\TM$ which is neither a prefix of $\TM$ nor a prefix 
of $\overline{\TM}$ and let  $\{a, b\} =  \{0,1\}$.
If $x$ is an occurrence of $aa$, then the $\#{\rm supp}(x)$ and $\#{\rm supp}(x+1)$ 
have the same parity, and hence $[x]_2$ ends in $1$ and the statement is verified for all 
factors beginning with $aa$.

We can then assume that $u$ begins with $ab$ and we will proceed by induction on $|u|$. Clearly the 
shortest such $u$ is of the kind $aba$. Then for each $x\in \TM_u,$ the number of $1$'s in the 
binary expansion of $x$ and of $x+2$ have the same parity. It follows that $[x]_2$ must 
end in $10$ or $11$ (and in fact it is easily verified that both are possible). Thus the result 
of the lemma is verified with  $k=0$.

Next suppose $|u|=N\geq 4$ and that the claim is true for all
factors $u$ of length smaller than $N$. If $u$ begins in $aba$,
then $\TM\big|_u\subseteq \TM\big|_{aba}$ and hence, as we have just
seen if $x\in \TM \big|_u$ we have that $[x]_2$ must end in either
$11$ or $10$. Otherwise $u$ must begin in either $0110$ or $1001$.
In this case, let $v$ denote the longest prefix of $u$ which is
either a prefix of $\TM$ or of $\overline{\TM}$. Then we can write
$u=va\lambda$ where $v$ begins in either $0110$ or $1001,$  $a\in
\{0,1\}$ and $v\in \{0,1\}^*$.  Since both $v0$ and $v1$ are
factors of $\TM$ it follows that $v=\tau(v')$ for some $v'$
strictly shorter than $v$ such that $v'$ begins in $01$ or $10,$
and $v'a$ is a factor of $\TM$ which is neither a prefix of $\TM$
nor of $\overline{\TM}$. By the induction hypothesis we deduce that there
exists a $k$ such that for all $x'\in \TM\big|_{v'a}$ we have that
$[x']_2$ ends in either $110^{k}$ or in $10^{k+1}$. Moreover,
since every occurrence of $va$ in $\TM$  (and hence of $u$) is the
image of $\tau$ of an occurrence in $\TM$ of $v'a$ it follows that
if $x\in \TM\big|_u$ then $x=2x'$ for some $x'\in \TM\big|_{v'a}$.
Whence $[x]_2$ ends in either $110^{k+1}$ or in $10^{k+2}$. We
have thus proved that if $u$ is neither a prefix of $\TM$ nor of
$\overline{\TM}$ and $u$ begins in $ab$, then there exists a $k$ such
that for any $x\in \TM \big|_u$ either $[x]_2$ ends in $110^{k}$
or in $10^{k+1}$. If only one of these cases occurs, then clearly
property \ref{eqn:point1} holds and we are done. Assume then that
both cases occur, we need to prove that $u$ is a prefix of
$\tau^n(aba)$ for some nonnegative integer $n$ (i. e. that we are
in case \ref{eqn:point2}). It is not difficult to prove (given the
definition of $\tau$) that every factor of $\TM$ of length at
least $4$ either appears only in odd positions or only in even
positions. Since we are assuming that there exist $x, y\in \TM
\big|_u$ such that $[x]_2$ ends in $110^{k}$ and $[y]_2$ ends in
$10^{k+1}$, it must be $k>0$ and $u$ occurs only in even
positions. Again from the definition of $\tau$, it is easy to see
that if $|u|$ is odd, then there exists a unique letter $c$ such
that every occurrence of $u$ is followed by $c$. Hence there
exists a unique $\alpha \in \{ 0, 1, \varepsilon \}$ such that
$|u\alpha|$ is even and $\TM\big|_u = \TM\big|_{u\alpha}$. From
the uniformity of $\tau$, since $u\alpha$ is a factor of $\TM$ of
even length which appears only in even positions, there exists
$u'$ shorter than $u$ such that $\tau(u')=u\alpha$ and
$\TM\big|_{u\alpha}=\{2x, x \in \TM\big|_{u'}\}$. Hence, for each
$x\in \TM \big|_{u'}$ either $[x]_2$ ends in $110^{k-1}$ or in
$10^{k}$ and both cases actually happen, thus, by induction
hypothesis, $u'$ is a prefix of $\tau^n(aba)$ for some $n$ and $u$
is a thus a prefix of $\tau^{n+1}(aba)$.
\end{proof}

We are now able to easily prove item 3. of our main theorem.
First of all, let us observe that it is readily verified that
$\{3,15,18\}\subseteq \TM \big|_{010}$ and $\{35, 47, 82\}\subseteq
\TM \big|_{101}$ and hence $\{2^n \cdot 3,2^n \cdot 15,2^n \cdot
18\}\subseteq \TM \big|_{\tau^n(010)}$ and $\{2^n \cdot 35, 2^n
\cdot 47, 2^n \cdot 82\}\subseteq \TM \big|_{\tau^n(101)}$ which
proves that if $aba \in \{ 010, 101\}$, then $\TM
\big|_{\tau^n(aba)}$ are $2$-summable for every nonnegative
integer n. Clearly then, if $u$ is a prefix of some $\tau^n(aba)$,
$\TM \big|_{u}$ is $2$- summable as well.

Let $u$ be a factor of $\TM$. In case point \ref{eqn:point1} of the preceding lemma holds, 
we have that there exists a nonnegative integer $k$ such that  each $x\in \TM \big|_u$ 
ends $10^k$. But then for any $x,y\in  \TM \big|_u,$ it follows that 
$[x+y]_2$ ends in $0^{k+1}$ and hence $x+y\notin \TM\big|_u$.

Thanks to point \ref{eqn:point2} of the preceding lemma we have thus proved that 
$\TM\big|_u$ is $2$-summable if and only if $u$ is a prefix of $\tau^n(aba)$ for some 
$n$ and $a, b$ distinct letters (considering that prefixes of $\TM$ and of $\overline{\TM}$ are 
as well prefixes of $\tau^n(aba)$). We are left to prove that  if $u$ is neither 
a prefix $\TM$ nor a prefix of $\overline{\TM}$, then $\TM\big|_u$ is not $3$-summable. Of course, 
the statement is trivial if $\TM\big|_u$ is not $2$-summable, so, as observed before, 
we can assume that point \ref{eqn:point2} of Lemma \ref{thm:lemmaTM} holds, that 
is there exists $k$ such that $[x]_2$ ends in $100^k$ or $110^k$  for each $x\in \TM_u$ 
and both cases happen.
Consider three points $x,y,z\in \TM \big|_u$. If $[x]_2$ ends in $100^k$, 
then $[x+y]_2$ ends in $000^k$ or $010^k;$ in either way it follows that 
$x+y \notin \TM \big|_u$. On the other hand if $[x]_2$ and $[y]_2$ both end in $110^k$, 
then $[x+y]_2$ ends in $100^k$ and hence as above it follows that $x+y+z \notin \TM\big|_u$. It 
follows that $\TM \big|_u$ is not $3$-summable, and the statement is complete.
\end{proof}

\noindent\textbf{Remark.} We proved part 1 of Theorem \ref{main} directly using the
numeration system, though it actually follows from parts 2, 3, and
the Finite Sums Theorem \cite {H}.

\noindent As a corollary of Theorem \ref{main} 2. and
Lemma \ref{TM1}  we obtain:

\begin{corollary}\label{notpartreg} $\Sigma^{\infty}$ is not partition regular, i.e., 
there exists a set $A\subseteq \nats$ which is infinite FS-big and a partition of 
$A=A_0\cup A_1$ such that neither $A_i$ is $2^\infty$-summable. \end{corollary}

One of the two referees suggested the following alternative proof of Corollary~\ref{notpartreg}. For each positive integer $n,$ let ${\rm supp}_3(n)$ denote the support of the ternary expansion of $n.$ Let $D=FS(\langle 3^n\rangle_{n\in \nats}).$ We note that for any $m,n \in D, $ if $m+n\in D$ then ${\rm supp}_3(m) \cap{\rm supp}_3(n)=\emptyset.$ Let $(E_i)_{i=1}^\infty$ be a partition of $\nats$ into infinite disjoint sets. For each $i\in \nats$ set
\[D_i=\{n \in D: {\rm supp}_3(n) \subseteq E_i\,\,\mbox{and}\,\, \#{\rm supp}_3(n) \leq i\}\]
and put $A=\bigcup_{i=1}^\infty D_i.$ Then  clearly $A \in \Sigma^{\infty}.$ However, let $B_0$ and $B_1$ be a partition of $\nats$ so that no two integers in $\nats$ of the form $x$ and $2x$ are both in $B_0$ or both in $B_1.$ For $i\in \{0,1\},$ put  \[A_i=\{n\in A: \#{\rm supp}_3(n) \in B_i\}.\] It is clear we cannot have $x,y \in A_i$ satisfying $x+y\in A_i$ and $\#{\rm supp}_3(x)=\#{\rm supp}_3(y).$ Thus $A_i$ is not $2^\infty$-summable.\\

We next derive two additional consequences of Theorem~\ref{main}. For this purpose, we recall some terminology which will be needed.
Let $\A$ be a finite non-empty set, and let $\A^\omega$ denote the set of all 
right infinite words $(x_n)_{n\in \nats}$ with $x_n\in \A.$  We endow $\A^\omega$ 
with the topology generated by the metric
\[d(x, y)=\frac 1{2^n}\,\,\mbox{where} \,\, n=\min\{k :x_k\neq y_k\}\] 
whenever $x=(x_n)_{n\in \nats}$ and $y=(y_n)_{n\in \nats}$ are two elements of $\A^\omega$. 
(This is also the product topology when $\A$ has the discrete topology.)
Let $T:\A^\omega \rightarrow \A^\omega$ denote the {\it shift} transformation defined by $T: (x_n)_{n\in \nats}\mapsto (x_{n+1})_{n\in \nats}.$ A point $x\in X$ is said to be {\it uniformly recurrent} in $X$ if for every neighborhood $V$ of $x$ the set
$\{n\,|\, T^n(x)\in V\}$ is syndetic, i.e., of bounded gap. Two points $x,y\in \A^\omega$ are said to be {\it proximal} if for every $\epsilon >0$ there exists $n\in \nats$ such that $d(T^n(x),T^n(y))<\epsilon.$ 

Let $X$ be a closed and $T$-invariant subset of $\A^\omega;$ the pair $(X,T)$ is 
called a {\it subshift} of $\A^\omega$. A subshift $(X,T)$ is said to be {\it minimal}
whenever $X$ and the empty set are the only $T$-invariant closed subsets of $X$.  
 To each $x\in \A^\omega$ is associated the subshift $(\Omega(x),T)$ where $\Omega(x)$ is the shift 
orbit closure of $x$.
A point $x\in \A^\omega$ is called {\it distal}  if the only point in $ \Omega(x)$ proximal to $x$ is $x$ itself.
If $x\in \A^\omega$ is uniformly recurrent, then the associated subshift $(\Omega(x),T)$ is minimal.
And, if $(\Omega(x),T)$ is minimal, then every point of $\Omega(x)$ is uniformly recurrent.  (For the
proofs of the last two assertions see for example \cite[Theorems 1.17 and 1.15]{refF}.)
It is well known that the Thue-Morse word is uniformly recurrent. (See for example
\cite[p. 832]{MoHe1}.)

As an application of Theorem~\ref{main} we have the following corollary.  In the proof of 
this corollary we use some facts from \cite{HS} about the algebraic structure of the Stone-\v Cech compactification
$\beta\nats$ of $\nats$, the points of which are the ultrafilters on $\nats$. Given an ultrafilter 
$p\in\beta\nats$ and a sequence $\langle x_n\rangle_{n=1}^\infty$ in a compact Hausdorff space $X$,
$p$-$\displaystyle\lim_{n\in \nats}x_n$ is the unique point $y\in X$ with the property that for
every neighborhood $U$ of $y$, $\{n\in\nats|x_n\in U\}\in p$.

\begin{corollary}\label{distal} The Thue-Morse word $\TM$ is distal. In particular, for each $n\geq 0,$  exactly one of the sets
$\{T^n(\TM)\big|_0, T^n(\TM)\big|_1\}$ is an IP-set.
\end{corollary}

\begin{proof} Suppose $x \in \Omega(\TM)$  is proximal to $\TM$.
Then, since $\TM$ is uniformly recurrent, we have 
by \cite[Theorem 19.26]{HS} that there exists a (minimal) idempotent ultrafilter 
$p \in \beta \nats$ with $p$-$\displaystyle\lim_{n\in \nats}T^n(\TM)=x$.
Given a prefix $u$ of $x$, $U=\{y\in\A^\omega|u$ is a prefix of $y\}$ is a 
neighborhood of $x$ so $\{n\in\nats|T^n(\TM)\in U\}\in p$; that is $\TM\big|_u \in p$.
Therefore by  \cite[Theorem 5.12]{HS}
$\TM\big|_u$ is an IP-set. By Theorem~\ref{main} it follows that $u$ is a prefix of 
$\TM$ and hence $x =\TM$ as required. Having established that $\TM$ is distal, it follows that 
$T^n(\TM)$ is distal for each $n\geq 0$. Finally, let us fix $n\geq 0$, and let $a\in \{0,1\}$ 
denote the initial symbol of $T^n(\TM).$ We claim that $T^n(\TM)\big|_a$ is an 
IP-set while $T^n(\TM)\big|_{\overline a}$ is not, where $\overline{a}:=1-a.$ 
Since $T^n(\TM)$ is uniformly recurrent, it follows from 
\cite[Theorem 19.23]{HS} that there exists a  idempotent ultrafilter 
$p \in \beta \nats$ with $p$-$\displaystyle\lim_{m\in \nats}T^m(T^n(\TM))=T^n(\TM)$.
Then as above, $T^n(\TM)\big|_a\in p$ so by \cite[Theorem 5.12]{HS} we have that $T^n(\TM)\big|_a$ is an IP-set. 
Now suppose on the other hand that $T^n(\TM)\big|_{\overline a}$ is also an IP-set.
Then by \cite[Theorem 5.12]{HS} there exists an idempotent $q \in \beta \nats$ 
such that $T^n(\TM)\big|_{\overline a}\in q$. We claim that
$q$-$\displaystyle\lim_{m\in \nats}T^m\big(T^n(\TM)\big)$ is proximal to $T^n(\TM)$ for which it
suffices by \cite[Lemma 19.22]{HS} to show that
$q$-$\displaystyle\lim_{r\in \nats}T^r\big( q$-$\displaystyle\lim_{m\in \nats}T^m\big(T^n(\TM)\big)\big)=
q$-$\displaystyle\lim_{k\in \nats}T^k\big(T^n(\TM)\big)$. To this end

\begin{align*}\displaystyle
q\hbox{-}\lim_{r\in \nats}T^r\big( q\hbox{-}\lim_{m\in \nats}T^m\big(T^n(\TM)\big)\big)&=\displaystyle
q\hbox{-}\lim_{r\in \nats}q\hbox{-}\lim_{m\in \nats}T^{r+m}\big(T^n(\TM)\big)\hbox{ by \cite[Theorem 3.49]{HS}}\\
&=\displaystyle (q+q)\hbox{-}\lim_{k\in \nats}T^{k}\big(T^n(\TM)\big)\hbox{ by \cite[Theorem 4.5]{HS}}\\
&=\displaystyle q\hbox{-}\lim_{k\in \nats}T^{k}\big(T^n(\TM)\big)\,.\end{align*}

Since $q$-$\displaystyle\lim_{m\in \nats}T^m\big(T^n(\TM)\big)$ is proximal to $T^n(\TM)$ and
$T^n(\TM)$ is distal, we have that $q$-$\displaystyle\lim_{m\in \nats}T^m(T^n(\TM))=T^n(\TM)$.
Thus, $T^n(\TM)\big|_a\in q$ from which it follows that
$\emptyset=T^n(\TM)\big|_a \cap T^n(\TM)\big|_{\overline a} \in q$, a contradiction.
\end{proof}

\begin{corollary} Let $N$ be a positive integer and set $x=t_Nt_{N-1}\ldots t_0\TM \in \Omega(\TM)$
where $\TM=t_0t_1t_2\ldots.$ 
Consider the partition $\nats =A_0\cup A_1$ where  $A_0=x\big|_0$ and $A_1=x\big|_1.$
Then $A_i-n$ is an IP-set for each $i\in \{0,1\}$ and $0\leq n\leq N.$ On the other hand,
for each $n>N,$ exactly one of the sets $\{A_0-n,A_1-n\}$ is an IP-set.
\end{corollary}

\begin{proof} For $a\in \{0,1\}$ we put $\overline{a}=1-a.$ We first note that since $a\TM\in \Omega(\TM)$ for some $a\in \{0,1\},$  by iteratively applying the morphism $0\mapsto 01,$ $1\mapsto 10$ we have that
 both $t_nt_{n-1}\ldots t_0\TM\in \Omega(\TM)$ and $\overline{t}_n\overline{t}_{n-1}\ldots \overline{t}_0\TM \in \Omega(\TM)$ for each $n\geq 0.$ 
 Fix a positive integer $N$ and put $x=t_Nt_{N-1}\ldots t_0\TM$ and $y=\overline{t}_N\overline{t}_{N-1}\ldots \overline{t}_0\TM.$ Then for each $0\leq n\leq N$, we have that $T^n(x)$ and $T^n(y)$ are proximal and begin in distinct symbols. Whence applying \cite[Theorem 19.26 \& Theorem 5.12 ]{HS} 
we deduce that $A_0-n=T^n(x)\big|_0$ and $A_1-n=T^n(x)\big|_1$ are both IP-sets for $0\leq n\leq N.$ On the other hand, applying Corollary~\ref{distal} we see that for each $n>N,$ exactly one of the sets
$\{A_0-n, A_1-n\}$ is an IP-set.
\end{proof}

\end{document}